\documentclass[12pt]{amsart}
\usepackage{fullpage, amsmath,amscd,amssymb,amssymb,xy}
\xyoption{all}

\theoremstyle{plain}
\newtheorem{theorem}{Theorem}[section]
\newtheorem{prop}[theorem]{Proposition}
\newtheorem{lemma}[theorem]{Lemma}

\newtheorem{claim}[theorem]{Claim}

\theoremstyle{definition}
\newtheorem{defn}{Definition}
\newtheorem{remark}{Remark}

\newcommand{\mat}[4]{
    \left[\begin{matrix}
      #1 & #2 \\
      #3 & #4
    \end{matrix}\right]
}
\newcommand{\Mat}[9]{
    \left[\begin{matrix}
      #1 & #2 & #3 \\
      #4 & #5 & #6 \\
      #7 & #8 & #9
    \end{matrix}\right]
}
\newcommand{\kk}{\mathbf{k}}
\newcommand{\m}{\mathfrak{m}}

\newcommand{\M}{\text{M}}
\newcommand{\GL}{\text{GL}}
\newcommand{\SL}{\text{SL}}

\newcommand{\Z}{\mathbb{Z}}

\newcommand{\jj}{\jmath}
\newcommand{\ii}{\imath}
\newcommand{\len}{\ell}

\DeclareMathOperator{\tr}{tr}

\title{Similarity classes of $3\times 3$ matrices\\ over a local principal ideal ring}

\author{Nir Avni}

\author{Uri Onn}
\thanks{The second author was supported by the Edmund Landau Minerva
Center for Research in Mathematical Analysis and Related Areas,
sponsored by the Minerva Foundation (Germany) and by ISF and BSF (US-Israel)}

\author{Amritanshu Prasad}

\author{Leonid Vaserstein}

\begin{document}
\maketitle \markboth{\textsc{NIR AVNI, URI ONN, AMRITANSHU PRASAD
AND LEONID VASERSTEIN}}{\textsc{SIMILARITY CLASSES OVER A LOCAL
RING}}

\begin{abstract} In this paper similarity classes of three by three matrices over a local principal ideal commutative ring are analyzed. When the residue field is finite, a generating function for the number of similarity classes for all finite quotients of the ring is computed explicitly.
\end{abstract}

\section{Introduction}
\label{sec:intro}

\subsection{Overview}
\label{sec:overview}

Let $A$ be a local principal ideal commutative ring. Let $\m$ denote
its maximal ideal and let $\kk=A/\m$ be the residue field. Let
$\len \in \mathbb{N} \cup \{\infty\}$ denote the length of $A$, that is, the
smallest positive integer for which $\m^\len=0$. Denote by $\M_n(A)$ and $\GL_n(A)$, the ring of matrices over $A$ and its group of units, respectively.

\begin{defn}
  Two matrices $\alpha$ and $\alpha'$ in $\M_n(A)$ are called \emph{similar} if there
  exists a matrix $X\in \GL_n(A)$ such that $X\alpha=\alpha' X$.
  The similarity classes of invertible matrices are the \emph{conjugacy classes} in $\GL_n(A)$.
\end{defn}

The classification problem of similarity classes in $n \times n$ matrices over rings has been considered by several authors and is considered to be a highly nontrivial quest, unless the ring in hand happens to be a field. For example, in \cite[\S4]{MR0498881} it is proved that already for $A=\Z/p^2\Z$, the classification of similarity classes in $\M_{4n}(A)$ contains the matrix pair problem in $\M_n(\Z/p\Z)$. The aim of this paper is to classify similarity classes in $\M_3(A)$ and $\GL_3(A)$.

\medskip

In order to put things into perspective we shall now take a short excursion in some known results on similarity classes. The similarity classes of matrices with entries in a field have been well understood in terms of their \emph{rational canonical forms} for a long time and are described, for example, by Dickson in \cite[Chapter V]{MR0105380}.

\smallskip

Over rings, only partial results are available;
In \cite{MR0220757}, Davis has shown using Hensel's method, that two matrices in $\M_n(\mathbb{Z}/p^\len\mathbb{Z})$ which are zeroes of a common polynomial whose reduction modulo $p$ has no repeated roots are similar if and only if their reductions modulo $p$ are similar.
In a similar vein, using an extension of the Sylow theorems (attributed to P.~Hall), Pomfret \cite{MR0309963} has shown that for a finite local ring, invertible matrices whose orders are coprime to the characteristic of the residue field are similar if and only if their images in the residue field are similar.

\smallskip

In \cite{MR579942}, Grunewald has given an algorithm for determining whether two matrices in $\GL_n(\mathbb{Q})$ are conjugate by an element of $\GL_n(\mathbb{Z})$.
For the special case where $n=3$, Appelgate and Onishi \cite{MR656422} have given a simpler algorithm to solve the same problem.
Given any two matrices $\alpha$ and $\alpha'$ in $\SL_n(\mathbb{Z}_p)$, Appelgate and Onishi \cite{MR681827} have given an explicit method to determine a positive integer $\len$ such that $\alpha$ and $\alpha'$ are conjugate in $\SL_n(\mathbb{Z}_p)$ if and only if they are conjugate in $\SL_n(\mathbb{Z}/p^\len\mathbb{Z})$, thereby reducing the conjugacy problem in the uncountable group $\SL_n(\mathbb{Z}_p)$ to a finite one.

\smallskip

In \cite{MR731899}, Nechaev has classified the similarity classes in the case $n=3$ and $\len = 2$. Close to the present article is \cite{MR714872}, where Pizarro has given a set of  representatives of the similarity classes in $\M_3(A)$ modulo scalar shift, when $A$ is a finite quotient of a complete discrete valuation ring. These representatives, however, do not lend themselves to the explicit enumeration of similarity classes when $A$ is finite, which is one of the main goals of this paper. Such explicit classification and enumeration has two implications in representation theory. The first is that together with the orbit method for $p$-adic Lie groups \cite{MR0579176}, the aforementioned classification is an important ingredient in computing the representation zeta function of $\SL_3(\mathcal{O})$, where $\mathcal{O}$ is the ring of integers of a $p$-adic field, see \cite{AO07}. The second is a positive indication that the isomorphism type of the group algebra $\mathbb{C}\GL_n(A)$ whenever $A$ is finite, depends only on $\kk$, as conjectured in \cite{ranktwo}, since we prove that $\dim_{\mathbb{C}} Z\left( \mathbb{C}\GL_n(A) \right)=|\text{Sim}(\GL_n(A))|$ depends only on $\kk$ for $n \le 3$.

\subsection{Some notation}
 Throughout we fix a uniformizing element $\pi \in \m$. For each $a\in A$
there is a unique integer $0\leq v(a)\leq \len$, called the valuation
of $a$, such that $a$ can be written as the product of $\pi^{v(a)}$
and a unit. For $1 \le \ii \le \len$ we write $A_\ii$ for the quotient $A/\m^\ii$ and $A_\ii^{\times}$ for its units. We fix a section $\kk=A_1 \hookrightarrow A$ which maps zero to zero with image $K_1 \subset A$. We can then define compatible sections $A_{\ii}\hookrightarrow
A$ for all $1\leq \ii < \len$, identifying $A_\ii$ with $K_{\ii}=\{\sum_{j=0}^{\ii-1}a_j \pi^j~|~a_j \in K_1\} \subset A$ as sets. We also have canonical identifications $A_{\len - j} \to \pi^jA$. The main examples of $A$ that we have in mind are the rings of integers of local fields and their finite
length quotients.

A shorthand notation is used for some commonly occurring matrices.
The identity matrix is denoted by $I$. The symbol $x^{ij}$ is used
to denote the matrix $I+xE^{ij}$, where $E^{ij}$ is the elementary
matrix with all entries zero, except for a \lq$1$\rq~in the $i^{\text{th}}$
row and $j^{\text{th}}$ column. Given a polynomial
$f(x)=x^n-a_{n-1}x^{n-1}-\cdots-a_1x-a_0$, its {\em companion matrix} is
the matrix
\begin{equation*}
  C_f=C(a_0,\ldots,a_{n-1})=\left[
    \begin{array}{ccccc}
      0 & 1 & 0 & \cdots & 0\\
      0 & 0 & 1 & \cdots & 0\\
      \vdots & \vdots & \vdots & \ddots & \vdots\\
      0 & 0 & 0 & \cdots & 1\\
      a_0 & a_1 & a_2 & \cdots & a_{n-1}
    \end{array}
  \right]. \qquad \qquad \qquad \qquad \qquad
\end{equation*}

The characteristic polynomial of the companion matrix $C_f$ is $f(x)$. Also, recall that a companion matrix represents a cyclic endomorphism, i.e. there exist $v \in A^n$ such that $\{C^i v~|~0 \le i <n \}$ is a basis for $A^n$. A block diagonal matrix is
denoted by
\begin{equation*}
  D(d_1,\ldots,d_n)
\end{equation*}
where $d_1,\ldots,d_n$ are the diagonal entries, which may
themselves be square matrices or scalars. Finally, another special
matrix that will come up often in this paper is
\begin{equation*}
  \quad E(m,a,b,c,d)=\Mat 0{\pi^m}0001abc +dI, \quad m \in \mathbb{N}, a,b,c,d \in A.
\end{equation*}
The special case $E(\len,0,0,c,d)$ will play an important role and will be denoted $J(c,d)$ for simplicity.

\subsection{Acknowledgements} The first and second authors thank Alex Lubotzky for supporting this research. The authors thank the referee for suggesting some improvements to the first draft of this paper.


\section{A baby version: $2 \times 2$ matrices}
In lack of an adequate reference with complete results (partial results can be found in \cite{MR0499009}), and since it serves as a solid basis for the reasoning in $n=3$, we now describe the similarity classes in the case $n=2$.

\subsection{Representatives}
Similarity classes in $\M_n(A)$ and $\GL_n(A)$ for $n=2$ are considerably easier to tackle than for $n>2$. The underlying reason is the following dichotomy:

\medskip

\centerline{\em Any element in $\M_2(\kk)$ is either scalar or cyclic.}

\medskip

\begin{lemma}\label{canonical22} Any element $\alpha \in \M_2(A)$ can be written in the form
\[
\alpha=dI+\pi^\jj \beta
\]
with $\jj \in \{0,\ldots,\len\}$ maximal such that $\alpha$ is congruent to a scalar matrix modulo $\m^\jj$, with unique $d \in K_\jj$ and unique $\beta \in \M_2(A_{\len-\jj})$ cyclic.
\end{lemma}

\begin{proof} The only fact that is not straightforward is that $\beta$ is cyclic. But the maximality of $\jj$ implies that $\beta$ is not a scalar modulo $\m$, and hence its image $\bar{\beta} \in \M_2(\kk)$ is cyclic. Using Nakayama's Lemma $\beta$ must be a cyclic as well.
\end{proof}
Clearly, $\alpha=dI+\pi^\jj\beta$ is similar to $\alpha'=d'I+\pi^\jj\beta'$ if and only if $d= d'$ and $\beta$ is similar to $\beta'$. Every cyclic matrix is conjugate to a companion matrix. Moreover, two
companion matrices are conjugate if and only if the polynomials
defining them are equal. We thus have
\begin{theorem}\label{theorem:2}
  For any $\alpha\in \M_2(A)$, let $\jj$, $d$ and $\beta$ be as above.
  Then $\alpha$ is similar to the matrix
  \begin{equation*}
    dI+\pi^\jj C(-\det(\beta), \tr(\beta)).
  \end{equation*}
  Thus $\jj\in \{0,\ldots,\len\}$, $d\in K_\jj$ and $\tr(\beta),\det(\beta)\in A_{\len-\jj}$ completely determine the similarity class of $\alpha$ in $\M_2(A)$. The similarity classes in $\GL_2(A)$ are represented by the subset of these elements such that $d \in A^{\times}$ for $\jj \ge 1$, or if $\jj=0$ then $\det(\beta) \in A^{\times}$.
\end{theorem}

\begin{remark}
In \cite{MR731899}, Nechaev has introduced the notion of a \emph{canonically determined} matrix.
This is a matrix $\alpha$ whose similarity class is completely determined by its Fitting invariants. These are the ideals in $A[x]$ generated by the $m\times m$ minors of $xI-\alpha$ (thought of as a matrix in $\M_n(A[x])$), for $m=1,\ldots,n$.
Nechaev conjectured, and proved it in certain cases, that a matrix is canonically determined if and only if all these ideals are principal ideals. For $n=2$, the ideal generated by the entries of $xI - \alpha$ is $A(x-d)+A\pi^{\jj}$ where $d$ and $\jj$ are as in Lemma \ref{canonical22}.
When $\len = \infty$, the characteristic polynomial of $\alpha$ determines
the characteristic polynomial of $\beta$. By Theorem \ref{theorem:2}, it follows that
the Fitting invariants determine the similarity class of $\alpha$,
which refutes an extension of his conjecture to the case $\len = \infty$. Representing $A$ as a factor ring of a discrete valuation domain, Kurakin \cite{MR2278884} has refined the Fitting invariants.

\end{remark}

\subsection{Enumeration}

Assume now that the residue field $\kk$ is finite of cardinality $q$. We wish to count the number of similarity classes in $\M_2(A_\ii)$ and $\GL_2(A_\ii)$. Although one can count them directly using Theorem \ref{theorem:2}, we shall use a recursive approach which will be very useful later on. Let $\eta:\M_n(A_{\ii+1}) \to M_n(A_\ii)$ denote the reduction map. Then, for any similarity class $\Omega \subset \M_n(A_\ii)$, the inverse image $\eta^{-1}(\Omega)$ is a disjoint union of similarity classes in $\M_n(A_{\ii+1})$. For $n=2$ the branching rules are the following. Let $a_\ii$ denote the number of similarity classes which are scalar matrices and let $b_\ii$ denote the number of the other similarity classes. We wish to establish a recursive relation between $(a_{\ii+1},b_{\ii+1})$ and $(a_{\ii},b_{\ii})$. Scalar matrices in $\M_2(A_{\ii+1})$ necessarily lie over scalar matrices $\M_2(A_\ii)$, hence $a_{\ii+1}=q a_{\ii}$. The non-scalar similarity classes in $\M_2(A_{\ii+1})$ can come from two sources; they can either lie over a scalar matrix in $\M_2(A_\ii)$, in which case they are $q^2a_{\ii}$ in number, or, they can lie over a non-scalar similarity class in $\M_2(A_\ii)$, in which case they are $q^2b_{\ii}$ in number (both assertions follow from Theorem \ref{theorem:2}). We therefore have

\[
\left[\begin{matrix} a_{\ii+1} \\ b_{\ii+1} \end{matrix}\right]= \left[\begin{matrix} q & 0 \\ q^2 & q^2 \end{matrix}\right] \left[\begin{matrix} a_{\ii} \\ b_{\ii} \end{matrix}\right].
\]
The initial values are
\[
v_{\M_2}=\left[\begin{matrix} a_1 \\ b_1 \end{matrix}\right]=\left[\begin{matrix} q \\ q^2 \end{matrix}\right] \qquad \text{and} \qquad v_{\GL_2}=\left[\begin{matrix} a_1 \\ b_1 \end{matrix}\right]=\left[\begin{matrix} q-1 \\ q^2-q \end{matrix}\right].
\]
Setting $T = \big[\begin{smallmatrix} q & 0 \\ q^2 & q^2 \end{smallmatrix}\big]$, we get that
\[
T^{\ii}=\left[\begin{matrix} q^\ii & 0 \\q^{\ii+1}\frac{q^\ii-1}{q-1} & q^{2\ii} \end{matrix}\right],
\]
hence
\[
\begin{split}
|\text{Sim}\left(\M_2(A_\ii)\right)|&=\epsilon T^{\ii-1}v_{\M_2}= (q^{2\ii+1}-q^\ii)/(q-1)\\
|\text{Sim}\left(\GL_2(A_\ii)\right)|&=\epsilon T^{\ii-1}v_{\GL_2}=q^{2\ii}-q^{\ii-1}
\end{split}
\]
where $\epsilon$ is the row vector $(1,1)$.

\section{Representatives for similarity classes of $3\times 3$ matrices}
\label{sec:3}

\subsection{Similarity classes over a field}
\label{sec:field}

The similarity classes over a field are given by their \emph{rational canonical forms} \cite[Chapter 7, Theorem 5]{MR0276251}:
\begin{theorem}[Similarity classes in $M_3(\kk)$]\label{theorem:field}
  Every matrix in $M_3(\kk)$ is similar to exactly one of the following:
  \begin{enumerate}
  \item\label{type:scalar} A scalar matrix $aI$, with $a\in \kk$ .

  \item\label{type:decompnoncyc} A matrix of the form $D(a,b,b)$, with $a,b\in \kk$ distinct.

  \item\label{type:indecompnoncyc} A matrix of the form
  \[
  \Mat a000a100a, \quad a\in \kk.
  \]

  \item\label{type:cyclic} A companion matrix of the form
  \[
  \Mat 010001abc, \quad a, b, c \in \kk.
  \]

  \end{enumerate}
\end{theorem}

\subsection{Some reductions}
\label{sec:reductions}
To begin with, we assign some invariants to similarity classes.
\begin{prop}
  \label{prop:central}
For any $\alpha\in \M_3(A)$ let $\jj\geq 0$ be the largest integer for
which $\alpha$ is congruent to a scalar matrix modulo $\m^\jj$. Then
$\alpha$ can be written, in a unique manner, as
$\alpha=dI+\pi^\jj \beta$, where $d\in K_\jj$ and $\beta \in
\M_3(A_{\len-\jj})$ is a matrix that is not congruent to a scalar matrix
modulo $\m$. Moreover, two such matrices
$\alpha_1=d_1I+\pi^\jj \beta_1$ and $\alpha_2=d_2I+\pi^\jj \beta_2$ are
similar if and only if $d_1=d_2$ and $\beta_1$ is similar to
$\beta_2$.
\end{prop}
It follows that the assignment $\alpha \mapsto \jj(\alpha)$ is a similarity invariant. Writing $\alpha = dI+\pi^{\jj}\beta$, we are therefore reduced to classifying similarity classes of matrices $\beta$ which are not congruent to a scalar modulo $\m$, that is, of types \eqref{type:decompnoncyc}, \eqref{type:indecompnoncyc} or \eqref{type:cyclic} in Theorem \ref{theorem:field}.

If \eqref{type:cyclic} occurs, then arguing as in Lemma \ref{canonical22}, it follows that $\beta$ is cyclic, hence determined by its characteristic polynomial. If \eqref{type:decompnoncyc} occurs we are essentially reduced to the $n=2$ case.

\begin{prop}
  \label{prop:not_so_hard}
  Suppose that the reduction of $\beta$ modulo $\m$ is $D(\overline a,\overline b,\overline b)$, with $\overline a,\overline b\in \kk$ and $\overline a\neq \overline b$.
Then $\beta$ is similar to a unique matrix of the form
\begin{equation}
  \label{eq:not_so_hard_canonical_form}\tag{$*$}
  D(a,bI+\pi^{\jj}C(c,d)),
\end{equation}
where  $1\leq \jj\leq \len$, $c,d\in A_{\len-\jj}$, $a\in A$ and $b\in K_{\jj}$
with $a- \bar{a} \equiv b - \overline b \equiv 0 \pmod \m$.
\end{prop}

\begin{proof} Denote the entries of $\beta$ by $\beta_{ij}$. The entries in places $(1,2)$ and $(1,3)$ of the conjugation of $\beta$ by a matrix of the form $I+xE^{12}+yE^{13}$ are
\begin{equation} \label{eq:lifting.conjugation.1}
\beta_{12}+(\beta_{22}-\beta_{11})x+\beta_{32}y-\beta_{12}x^2-\beta_{13}xy
\end{equation}
and
\begin{equation} \label{eq:lifting.conjugation.2}
\beta_{13}+(\beta_{33}-\beta_{11})y+\beta_{23}x-\beta_{13}y^2-\beta_{23}xy
\end{equation}
respectively.

Let $X$ be the scheme defined by the polynomials (\ref{eq:lifting.conjugation.1}) and (\ref{eq:lifting.conjugation.2}). By our assumptions, $\beta_{23},\beta_{32}\equiv 0 \pmod\m$ and $\beta_{22}-\beta_{11},\beta_{33}-\beta_{11}\not\equiv0\pmod\m$. Therefore, the point $(0,0)\in\kk^2$ is a non-singular point of $X\times_{\textrm{Spec(A)}}\textrm{Spec($\kk$)}$. By the Hensel lemma, it can be lifted to a point $(x_0,y_0)\in X(A)$. Conjugating $\beta$ by $I+x_0E^{12}+y_0E^{13}$, we can assume that $\beta_{12}=\beta_{13}=0$. Note that this conjugation does not change the reduction of $\beta$ modulo $\m$. Similarly, there are $x_1,y_1\in A$, such that conjugating $\beta$ by $I+x_1E^{21}+y_1E^{31}$ makes $\beta_{21}$ and $\beta_{31}$ equal to zero. Since this last conjugation does not change the entries in places $(1,2)$ and $(1,3)$, the result is a block diagonal matrix.

The classification of similarity classes for $2\times 2$ matrices (Theorem~\ref{theorem:2}) shows that $\beta$ is similar to a matrix of the kind in (\ref{eq:not_so_hard_canonical_form}).
  That no two distinct matrices of type (\ref{eq:not_so_hard_canonical_form}) are similar follows from Lemma~\ref{lemma:not_so_hard} below applied to $\beta-bI$.
\end{proof}
\begin{lemma}
  \label{lemma:not_so_hard}
   Let $B$ and $B'$ be two matrices in $\M_2(A)$ which are congruent to $0$ modulo $\m$, and let $a, a' \in A^{\times}$.
  Then the two block matrices
  \begin{equation*}
    \beta=\mat a00B \text{ and } \beta'=\mat{a'}00{B'}
  \end{equation*}
  are similar if and only if $a=a'$ and $B$ is conjugate to $B'$.
\end{lemma}
\begin{proof}
  Clearly, the condition for similarity in the statement of the lemma is sufficient.
  To see that it is necessary, suppose $X\in \GL_3(A)$ is such that $X\beta=\beta'X$. Write $X$ as a block matrix $\big(\begin{smallmatrix} x& y \\ z& W \end{smallmatrix}\big)$. Evaluation of the above equality in terms of block matrices gives
  \begin{equation*}
    \mat{xa}{yB}{za}{WB}=\mat{a'x}{a'y}{B'z}{B'W}.
  \end{equation*}
  Since $B\equiv 0 \mod \m$ and $a'$ is a unit, comparing the top right entries shows that $y\equiv 0 \mod \m$.
  Similarly, $z\equiv 0 \mod \m$.
  It follows that $x$ and $W$ are invertible, which implies that $a=a'$ and $B$ is similar to $B'$.
\end{proof}

It remains to analyze case \eqref{type:indecompnoncyc} of Theorem \ref{theorem:field}, to which we dedicate the next subsection.

\subsection{The hard case}
\label{sec:hard} Assume that $\beta \in \M_3(A)$ is such that its reduction modulo $\m$ is of the form
  \[
  J(0,\bar{d})= \Mat {\bar{d}}000{\bar{d}}100{\bar{d}}, \quad \bar{d}\in \kk.
  \]

\begin{prop} Any matrix $\beta \in \M_3(A)$ which lies above $J(0,\bar{d}) \in \M_3(\kk)$ is a conjugate of a matrix of the form
\[
E=E(m,a,b,c,d)=\left[\begin{matrix} 0 & \pi^m  & 0 \\ 0 & 0 & 1 \\ a & b & c \end{matrix}\right]+dI
\]
with  $m \ge 1$ and $a,b,c,d-\bar{d}\equiv 0 \pmod \m$.

\end{prop}

\begin{proof}

 Reduce $\beta$ to a matrix of the form $E(m,a,b,c,d)$ by the following sequence of similarity transformations:
  \begin{enumerate}
  \item conjugation of $\beta$ by a diagonal matrix makes the $(2,3)$-entry
    equal to $1$.
  \item conjugation by $(\beta_{13})^{12}$ kills the $(1,3)$-entry.
  \item addition of a scalar matrix kills the $(1,1)$-entry.
  \item conjugation by $(\beta_{21})^{31}$ kills the $(2,1)$-entry.
  \item conjugation by $(\beta_{22})^{32}$ kills the $(2,2)$-entry.
  \item conjugating by a diagonal matrix makes the $(1,2)$-entry of the form $\pi^m$ for some positive integer $m$.
  \end{enumerate}
\end{proof}

In order to avoid redundancies we should check when two representatives $E(m_1,a_1,b_1,c_1,d_1)$ and $E(m_2,a_2,b_2,c_2,d_2)$ are conjugate. This question is the core of the difficulty when passing from $n=2$ to $n=3$. A similar problem is considered in \cite{MR714872}. Note that the four types of {\em essentially cyclic} matrices defined by Pizarro are similar to some $E(m,a,b,c,d)$, in particular they are similar to each other. The parametrization given in \cite[Thm 2.17]{MR714872} is ineffective in the sense that it does not give a reasonable way of enumerating the
classes for finite rings. This is the main point in which we take an alternative route and focus
on the branching of classes of level $\ii$ in level $\ii+1$, which is the subject of the next section. Consequently,
we shall be able to enumerate the similarity classes in
the case of finite rings in Section 5.

\section{Analysis of the hard case}

Let $\alpha_1=E(m_1,a_1,b_1,c_1,d_1)$ and $\alpha_2=E(m_2,a_2,b_2,c_2,d_2)$ be in $\M_3(A_\ii)$. The elements $\alpha_1$ and $\alpha_2$ are similar if and only if there exist $X=(x_{ij}) \in \GL_3(A_\ii)$ which satisfies
\begin{equation}\label{Y}
Y:=\alpha_1X-X\alpha_2=0.
\end{equation}
Such a matrix $X$ should satisfy
\begin{equation}\label{YA}
\begin{split}
Y_{21}:& \quad x_{31}=a_2x_{23}+(d_2-d_1) x_{21} \qquad \qquad \qquad \qquad \qquad \qquad \qquad \\
Y_{22}:& \quad x_{32}= \pi^{m_2} x_{21}+b_2 x_{23} + (d_2-d_1) x_{22}\\
Y_{23}:& \quad x_{33}= x_{22}+(c_2-d_1+d_2)x_{23} \\
Y_{13}:& \quad x_{12}=\pi^{m_1} x_{23}-c_2x_{13}+ (d_1-d_2) x_{13}
\end{split}
\end{equation}
Using equations \eqref{YA} and the fact that $d_1-d_2,\pi
^{m_i},a_i,b_i,c_i$ are all congruent to $0$ mod $\m$, we get that
$\det(X)$ is congruent to $x_{11}x_{22}^2$ modulo $\m$. Therefore, we get
the extra condition
\begin{equation}\label{Ydet}
\det: \quad  x_{11},x_{22} \in A^{\times}. \qquad \qquad \qquad \qquad \qquad \qquad \qquad
\end{equation}
There are five additional equations
\begin{equation}\label{YB}
\begin{split}
Y_{11}:& \quad \pi^{m_1}x_{21}-a_2x_{13}+(d_1-d_2) x_{11}=0 \qquad \qquad \qquad \qquad \qquad \qquad \\
Y_{12}:& \quad \pi^{m_1}x_{22}-\pi^{m_2}x_{11}-b_2x_{13}+(d_1-d_2) x_{12}=0 \\
Y_{31}:& \quad a_1x_{11}+b_1x_{21}+c_1x_{31}-a_2x_{33}+(d_1-d_2)x_{31}=0 \\
Y_{32}:& \quad a_1x_{12}+b_1x_{22}+c_1x_{32}-\pi^{m_2}x_{31}-b_2x_{33}+(d_1-d_2) x_{32}=0\\
Y_{33}:& \quad a_1x_{13}+b_1x_{23}+c_1x_{33}-x_{32}-(c_2-d_1+d_2)x_{33}=0
\end{split}
\end{equation}
whose solution in general is very complicated. To this end, we shall narrow down possibilities by looking at the centralizers of representatives and then be able to solve these equations.

\subsection{Centralizers} In this section we shall take a closer look on the centralizers in $\GL_3(A_\ii)$ of the elements $E(m,a,b,c,d)$, which are subgroups that are defined over $A$. We shall use the Greenberg functor $\mathcal{F}$ \cite{MR0126449,MR0156855} which enables us to view them as algebraic groups over $\kk$. Taking $\alpha=\alpha_1=\alpha_2=E(m,a,b,c,d)$ in \eqref{Y} we get that $X=(x_{ij})$ commutes with $\alpha$ if and only if
\begin{equation}\label{equations}
\begin{split}
 \quad a x_{13}&=\pi ^m x_{21} \\
 \quad b x_{13}&=\pi^m (x_{22}-x_{11})\\
 \quad b x_{21}&=a(x_{22}-x_{11}),
\end{split}
\end{equation}
with $x_{12}, x_{31},x_{32}, x_{33}$ determined, respectively, by equations $Y_{13},Y_{21},Y_{22}, Y_{23}$ above, together with an additional free variable $x_{23}$.
The system of equations \eqref{equations} possess symmetries which can be made more transparent. Write $a=u_1 \pi^{t_1}$ and $b=u_2 \pi^{t_1}$ for some invertible elements $u_1,u_2 \in A_{\ii}^{\times}$ and replace $m$ by $t_3$. Then, the system \eqref{equations} can be written as

\begin{equation}\label{equations'}
\begin{split}
\text{(i)} \quad \pi^{t_1} y_3   &=\pi^{t_3} y_1, \\
\text{(ii)} \quad \pi^{t_2} y_3  &=\pi^{t_3} y_2,\\
\text{(iii)} \quad \pi^{t_2} y_1 &=\pi^{t_1}y_2,
\end{split}
\end{equation}
with new variables
\[
\begin{split}
&y_1=u_1^{-1}x_{21},\\
&y_2=u_2^{-1}(x_{22}-x_{11}),\\
&y_3=x_{13}.
\end{split}
\]

In order to solve these equations in $\M_3(A_\ii)$, we may assume, by relabeling the variables, that $t_1 \le t_2
\le t_3$. It then follows that equation (ii) can
be omitted. We are left with two equations
\[
\begin{split}
&y_2 = \pi^{t_2-t_1}y_1 \mod \pi ^{\ii-t_1} \\
&y_3 = \pi^{t_3-t_1}y_1 \mod \pi ^{\ii-t_1}
\end{split}
\]
This gives
\[
 \left \{   (y_1,y_2,y_3) ~\text{satisfying \eqref{equations'}}
\right \}  = A_{\ii} \times A_{t_1} \times A_{t_1},
\]
which using the Greenberg functor can be identified with the affine $\kk$-space $\mathbb{A}_\kk^{\ii+2t_1}$.
To ensure that $X \in \GL_3(A_\ii)$, we need $x_{11},x_{22}\in A_\ii
^{\times}$ by \eqref{Ydet}. The elements of the centralizer of $\alpha$ in $\GL_3(A_\ii)$, are identified under the Greenberg functor with the following $\kk$-varieties, depending on the values of the $t_i$'s, which we now relabel according to the original variables: $m$, $v(a)$ and $v(b)$. The centralizers depend on the relative value of $\min\{m,v(a)\}$ and $v(b)$ as follows
 \begin{enumerate}
 \item[]  $v(b) \le \min\{m,v(a)\}$: \qquad  $\mathcal{F}\left(\text{Stab}_{\GL_3(A_{\ii})}(\alpha)\right) \simeq \left(\mathbb{A}_{\kk}^{\times}\right)^2 \times  \mathbb{A}_{\kk}^{3\ii+2v(b)-2}$,
     \smallskip
     \item[]  $v(b)  > \min\{m,v(a)\}$: \qquad $\mathcal{F}\left(\text{Stab}_{\GL_3(A_{\ii})}(\alpha)\right) \simeq \mathbb{A}_{\kk}^{\times} \times  \mathbb{A}_{\kk}^{3\ii+2\min\{m,v(a)\}-1}$,
     \end{enumerate}
where $\mathbb{A}_\kk^{\times}=\mathbb{A}_{\kk} \smallsetminus \{0\}$ stands for the punctured affine line.
The difference between the two cases arises from \eqref{equations}, as $x_{11}$ and $x_{22}$ can be chosen independently if $v(b) \le \min\{m,v(a)\}$ but not if $v(b) > \min\{m,v(a)\}$.

\medskip

We shall now describe the branching rules of these representatives when passing from $\M_3(A_{\ii-1})$ and $\M_3(A_{\ii})$. We fix a compatible system of representatives in level $\ii$ which lie above their reduction in level $\ii-1$.

\begin{claim}\label{centralizer-types} Let $J(\bar{c},\bar{d})=E(\ii-1,0,0,\bar{c},\bar{d}) \in M_3(A_{\ii-1})$ with $\bar{c} \equiv 0 \mod \m$. Then the following four types of representatives which lie above $J(\bar{c},\bar{d})$ in $\M_3(A_\ii)$ represent disjoint similarity classes.
\begin{align*}
& \quad \underline{\text{Type I}} &   & \qquad \underline{\text{Type II}} &  & \quad \underline{\text{Type III}_\epsilon}   \quad (\epsilon \in \{0,1\}) & &\\
& \left[\begin{matrix} 0 & 0  & 0 \\0 &0 & 1 \\ 0 & 0 & c \end{matrix}\right]+dI & & \left[\begin{matrix} 0 & \pi^{\ii-1}\epsilon  & 0 \\ 0 & 0 & 1 \\  \pi^{\ii-1}a &  \pi^{\ii-1}b & c \end{matrix}\right]+dI & & \left[\begin{matrix} 0 & \pi^{\ii-1}\epsilon  & 0 \\ 0 & 0 & 1 \\ \pi^{\ii-1}a & 0 & c \end{matrix}\right]+dI & & \\
& ~ & & a \in \kk, b \in \kk^{\times}, \epsilon \in \{0,1\} & & \qquad (a,\epsilon) \ne (0,0) & &
\end{align*}

with $c,d \in A_\ii$ such that $c - \bar{c} \equiv d- \bar{d} \equiv 0 \pmod {\m^{\ii-1}}$.
\end{claim}

\begin{proof} The $\kk$-varieties which correspond to each type are $\mathcal{F}(Z_I)\simeq\left(\mathbb{A}_\kk^{\times}\right)^2 \times \mathbb{A}_\kk^{5\ii-2}$, $\mathcal{F}(Z_{II})\simeq\left(\mathbb{A}_\kk^{\times}\right)^2 \times \mathbb{A}_\kk^{5\ii-4}$ and $\mathcal{F}(Z_{III})\simeq \mathbb{A}_\kk^{\times} \times \mathbb{A}_\kk^{5\ii-3}$, hence each type remains invariant under conjugation. In order to separate the subtypes $\text{III}_0$ and $\text{III}_1$ we use equation $Y_{12}=0$, which reads in this case $x_{11} \epsilon_1 \equiv x_{22} \epsilon_2 \pmod \m$, therefore if two matrices of type III are similar they must have the same $\epsilon$.

\end{proof}

\begin{remark} When $\kk$ is finite, Claim \ref{centralizer-types} can be proved without using the Greenberg machinery. We just note that if the size of $\kk$ is $q$, then the size of the stabilizer of the matrix is $(q-1)^2q^{5\ii-2},(q-1)^2q^{5\ii-4},(q-1)q^{5\ii-3}$ in cases I,II, and III respectively. Since no two of these numbers can be equal, types I,II,III are disjoint. In fact, one can show (see \cite{AO07} for complete details) that the centralizers of the types I,II and III modulo $\m$ are isomorphic to
\[
\mathrm{Aut}_{\kk[x]}(\kk[x]/(x^2) \oplus \kk), \quad \mathbb{G}_m(\kk)^2 \times \mathbb{G}_a(\kk) \quad \mathrm{and} \quad \mathbb{G}_m(\kk) \times \mathbb{G}_a(\kk)^2,
\]
respectively.

\end{remark}

\subsection{Sieving away redundancies}

The list of representatives in Claim \ref{centralizer-types} is exhaustive, since we covered all the possible lifts of $J(\bar{c},\bar{d})$. To make sure that there are no repetitions we need a more delicate analysis.

\begin{theorem} The following is an exhaustive list of non-similar elements lying above $J(\bar{c},\bar{d})$ (where $\bar{c} \equiv 0 \mod \m$)
\begin{enumerate}
\item [] $\mathrm{I}$ \qquad $~~E(\ii,0,0,c,d)$, with $d-\bar{d} \equiv c-\bar{c} \equiv 0 \pmod {\m^{\ii-1}}$.

\item [] $\mathrm{II}$ \qquad $E(\ii-1,0,b\pi^{\ii-1},c,d)$, with $c-\bar{c} \equiv d-\bar{d} \equiv 0 \pmod {\m^{\ii-1}}$, $b \in \kk^{\times}$.

\item [] $\mathrm{III}_0$ \quad $E(\ii,\pi^{\ii-1},0,c,0)$, with $c-\bar{c} \equiv 0 \pmod {\m^{\ii-1}}$.

\item [] $\mathrm{III}_1$ \quad $E(\ii-1,a\pi^{\ii-1},0,\bar{c},d)$, with $d-\bar{d} \equiv 0 \pmod {\m^{\ii-1}}$, $a \in \kk$.

\end{enumerate}

All in all, there is a bijection between the conjugacy classes of
elements lying above $J(\bar{c},\bar{d})$ and the set $\kk^2 \sqcup \kk^2\times
\kk^{\times} \sqcup \kk \sqcup \kk^2$ (corresponding to the classes
$\mathrm{I},\mathrm{II},\mathrm{III}_0,\mathrm{III}_1$ respectively).

\end{theorem}

\begin{proof} \textbf{Case I.} We claim that the matrices $\alpha_1=E(\ii,0,0,c_1, d_1)$ and $\alpha_2=E(\ii,0,0,c_2,d_2)$ in $\M_3(A _\ii)$, with $c_1-c_2\equiv d_1-d_2 \equiv 0 \pmod {\m^{\ii-1}}$ are $\GL_3(A _\ii)$-conjugate if and only if $d_1=d_2$ and $c_1=c_2$. Indeed, in this case $Y_{11}=(d_1-d_2) x_{11}$, and since $x_{11}$ is a unit, $d_1-d_2=0$. The equality of the $c_j$'s now follows from $c_j=\tr(\alpha_j)$.


\textbf{Case II.} We first claim that we may assume that $\epsilon=1$. This follows from the following
\begin{lemma}
If $v(b) \le \min\{m,v(a)\}$ then $E(m,a,b,c,d) \sim E(v(b),a\pi^{m-v(b)},b,c,d)$.
\end{lemma}
\begin{proof} The matrix
\[
X=\left[\begin{matrix} 1 &  -ec  &  e  \\
ea\pi^{-v(b)} & 1 & 0 \\ 0  &  ec & 1 \end{matrix}\right],
\]
where $e=b^{-1}(\pi^m-\pi^{v(b)})$, realizes the similarity. Note that although $b$ is not invertible, $e$ and $a\pi^{-v(b)}$ are well defined as $v(b) \le v(a),m$.
\end{proof}
We now claim that $\alpha=E(\ii-1,a\pi^{\ii-1},b\pi^{\ii-1},c,d)$ with $b \in \kk^{\times}$ is similar to a matrix of the form $\alpha'=E(\ii-1,0,b\pi^{\ii-1},c',d')$, i.e. that $a$ can be eliminated. Indeed, one checks that conjugating $\alpha$ with
\[
X_e=\left[\begin{matrix} 1 &  0  &  0  \\
-e & 1 & 0 \\ e^2\pi^{\ii-1}  &  -2e\pi^{\ii-1} & 1 \end{matrix}\right], \qquad e \in A_{\ii},
\]
gives
\[
X_e\alpha X_e^{-1}=E(\ii-1,(a-eb)\pi^{\ii-1},b\pi^{\ii-1},c-3e\pi^{\ii-1}, e\pi^{\ii-1}),
\]
and since $b$ is invertible, there exist a choice of $e$ such that $a-be \equiv 0 \pmod {\m}$. Summarizing the last two steps, we may assume that a matrix of type II can be conjugated to a matrix of the form  $E(\ii-1,0,b\pi^{\ii-1},c,d)$ with $b \in \kk^{\times}$. We check that two such matrices are similar if and only if they have the same $b$, $c$ and $d$. If $\alpha_1=E(\ii-1,0,b_1\pi^{\ii-1},c_1,d_1)$ and $\alpha_2=E(\ii-1,0,b_2\pi^{\ii-1},c_2,d_2)$ are similar and $(c_1-c_2) \equiv (d_1-d_2) \equiv (b_1-b_2) \equiv 0 \pmod {\m^{\ii-1}}$, then
\[
\begin{split}
&Y_{32}=0 ~\Longrightarrow~  b_1=b_2 \\
&Y_{31}=0 ~\Longrightarrow~ x_{21} \equiv 0 \pmod \m\\
&Y_{11}=0  ~\Longrightarrow~ d_1=d_2\\
&\tr(\alpha_1)=\tr(\alpha_2) ~\Longrightarrow~ c_1=c_2.
\end{split}
\]


\textbf{Case $\textbf{III}_0$.} First, observe that for any matrix $\alpha=E(\ii,a\pi^{\ii-1},0,c,d)$ of type $\mathrm{III}_0$ we may assume that $a=1$, since we know that $a$ is invertible, and by conjugating $\alpha$ with $D(a,1,1)$ we get $E(\ii,\pi^{\ii-1},0,c,d)$. Second, observe that the matrix
\[
 X=\left[\begin{matrix} 1 &  d^2\pi^{\ii-1}+dc  &  -d  \\
0 & 1 & 0 \\ 0  &  d\pi^{\ii-1} & 1 \end{matrix}\right],
\]
conjugates $E(\ii,\pi^{\ii-1},0,c,d\pi^{\ii-1})$ to $E(\ii,\pi^{\ii-1},0,c+3d\pi^{\ii-1},0)$. It follows that any matrix of type $\mathrm{III}_0$ is equivalent to a matrix of the form $E(\ii,\pi^{\ii-1},0,c,d_0)$, where $d_0 \in A_{\ii}$ is a fixed element which lies above $\bar{d}$, and $c$ varies among the lifts of $\bar{c}$. Since $\tr(\alpha)=c+3d_0$, all these elements are distinct and the assertion is proved.


\textbf{Case $\textbf{III}_1$.} We claim that $\alpha_1=E(\ii-1,\pi^{\ii-1}a_1,0,c_1,d_1)$ and $\alpha_2=E(\ii-1,\pi^{\ii-1}a_2,0,c_2,d_2)$, with $c_1 - c_2 \equiv d_1-d_2 \pmod {\m^{\ii-1}}$ are $\GL_3(A _\ii)$-conjugate if and only if their traces are equal and $a_1=a_2$.
To get the only if part, we use equation $Y_{12}=0$ to deduce that $x_{11} \equiv x_{22} \pmod \m$, and substituting the latter equality into $Y_{31}=0$ gives that $a_1=a_2$. The equality of the traces is of course a necessary condition as well. To prove the if part, assuming equality of traces $c_2=c_1+3\pi^{\ii-1}(d_1-d_2)$ and that $a_1=a_2$, the matrix
\[
X=\left [ \begin{matrix} 1 & 0  & 0 \\ -\delta  & 1 & 0 \\ \pi ^{\ii-1}\delta^2 & -2\pi ^{\ii-1}\delta  & 1 \end{matrix} \right ],
\]
where $\delta=d_1-d_2$, is an invertible solution to the equation $\alpha_1X=X\alpha_2$.

\end{proof}

\subsection{The branching rules}

Classes of types I, II and III branch in the following way when increasing the level
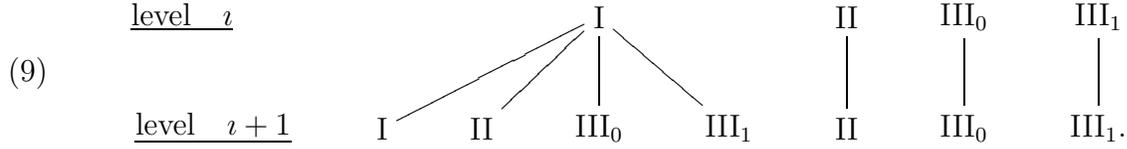
\begin{equation}\label{branchingJ}
 \begin{matrix}\xymatrix{ \underline{\mathrm{level} \quad \ii}\qquad  &  & & \mathrm{I}\ar@{-}[d]\ar@{-}[dll] \ar@{-}[dl] \ar@{-}[dr]& & \mathrm{II}\ar@{-}[d]  &  \mathrm{III}_0\ar@{-}[d] & \mathrm{III}_1\ar@{-}[d] \\
  \underline{\mathrm{level} \quad \ii+1}  &  \mathrm{I} & \mathrm{II}   &  \mathrm{III}_0 & \mathrm{III}_1   & \mathrm{II}   &  \mathrm{III}_0 & \mathrm{III}_1.
  }\end{matrix}
  \end{equation}
This follows from the fact that the relation between $v(b)$ and $\min\{v(a),m\}$ remains unchanged when the level increases for types II and III. Moreover, from Claim \ref{centralizer-types} it follows that types II and III lift in a \lq regular\rq~  fashion. Namely, for each similarity class
$C$ in $\M_3(A_\ii)$, the set of similarity classes in $\M_3(A_{\ii+1})$ that
lie over $C$ is in bijection with $\kk^3$, which is the same behavior as of cyclic elements.

\section{Enumeration of similarity classes} We now specialize to the case where $\kk$ is a finite field with $q$ elements, and count the number of similarity classes in $\M_3(A_{\ii})$ and in $\GL_3(A_{\ii})$ for all $\ii \in \mathbb{N}$. Given a matrix $\alpha \in \M_3(A_\ii)$, recall that $\jj$ is the maximal integer such that $\alpha$ is scalar modulo $\m ^\jj$. Then modulo $\m ^{\jj+1}$, $\alpha$ is conjugate to a matrix $dI+\pi ^\jj \beta$ where modulo $\m$, $\beta$ is either cyclic, $D(a,b,b)$ with $a \ne b$ or equal to $J(0,e)$ . If $\beta \equiv D(a,b,b) \pmod \m$ then it is conjugate to a matrix in one of the following forms
\[
\left [ \begin{matrix} a & 0 & 0 \\ 0 & b & 0 \\ 0 & 0 & b \end{matrix} \right ] \qquad \text{or} \qquad  \left [ \begin{matrix} a & 0 & 0 \\ 0 & b & 0 \\ 0 & 0 & b \end{matrix} \right ] + \pi ^m \left [ \begin{matrix} 0 & 0 \\ 0 & C \end{matrix} \right ]
\]
where $C \in \M_2(A_{\ii-m})$ is cyclic and $\ii>m>\jj$. If $\beta$ equals $J(0,e)$ modulo $\m$ then it is of type I above or it lies over types $\text{II}$, $\text{III}_0$, $ \text{III}_1$ above. We divide the set of matrices to the following classes:
\begin{enumerate}
\item[(i)] $dI$.
\item[(ii)] $dI + \pi ^\jj D(a,b,b)$, $(a\neq b)$.
\item[(iii)] $dI +\pi ^\jj J(c,d)$.
\item[(iv)] The rest of the options.
\end{enumerate}

Before giving the precise analysis, observe right away that these possibilities are partially ordered: $\mathrm{(i)} \prec  \mathrm{(ii)}, \mathrm{(iii)},\mathrm{(iv)}$, $\mathrm{(ii)} \prec  \mathrm{(iv)}$, and $\mathrm{(iii)} \prec \mathrm{(iv)}$, in the sense that each possibility may branch only to possibilities greater or equal to it. Indeed,

\begin{itemize}
\item A similarity class of type (i) in level $\ii$ splits into all the other similarity classes in level $\ii+1$, these are precisely the similarity classes in $\M_3(\kk)$: there are $q$ similarity classes of type (i), $q^2-q$ similarity classes of type (ii), $q$ similarity classes of type (iii) (all are of the form $dI+\pi ^\ii J(0,e)$) and $q^3$ similarity classes of type (iv) (all of the form $dI+\pi ^\ii \beta$ where $\beta$ is cyclic).

\item A similarity class of type (ii) splits into similarity classes of type (ii) and (iv) only: there are $q^2$ similarity classes of type (ii) over it and $q^3$ similarity classes of type (iv) over it, all are of the form
\[
\left [ \begin{matrix} a & 0 & 0 \\ 0 & b & 0 \\ 0 & 0 & b \end{matrix} \right ] + \pi ^\ii \left [ \begin{matrix} 0 & 0 \\ 0 & C \end{matrix} \right ]
\]
with $C$ cyclic in $\M_2(\kk)$.
\item For type (iii) we already computed: there are $q^2$ classes of type (iii) and $(q^3-q^2)+q+q^2 = q^3+q$ classes of type (iv) above it.

    \item Although type (iv) contains various different subtypes, every conjugacy class of type (iv), splits into $q^3$ similarity classes of type (iv) in level $\ii+1$.

\end{itemize}
Altogether, the numbers $\eta^{\mathrm{i}}_\ii$, $\eta^{\mathrm{ii}}_\ii$ $\eta^{\mathrm{iii}}_\ii $ and $\eta^{\mathrm{iv}}_\ii$ of similarity classes of type (i),(ii),(iii) and (iv), respectively, in level $\ii$ satisfy the following recursion:
\[
\left [ \begin{matrix} \eta^{\mathrm{i}}_{\ii+1} \\ \eta^{\mathrm{ii}}_{\ii+1} \\ \eta^{\mathrm{iii}}_{\ii+1} \\ \eta^{\mathrm{iv}}_{\ii+1}\end{matrix} \right ] = T \left[\begin{matrix} \eta^{\mathrm{i}}_\ii \\ \eta^{\mathrm{ii}}_\ii \\ \eta^{\mathrm{iii}}_\ii \\ \eta^{\mathrm{iv}}_\ii \end{matrix} \right ]= \left [ \begin{matrix} q & 0 & 0 & 0 \\ q^2-q & q^2 & 0 & 0 \\ q & 0 & q^2 & 0 \\ q^3 & q^3 & q^3+q & q^3 \end{matrix} \right ]  \left[\begin{matrix} \eta^{\mathrm{i}}_\ii \\ \eta^{\mathrm{ii}}_\ii \\ \eta^{\mathrm{iii}}_\ii \\ \eta^{\mathrm{iv}}_\ii \end{matrix} \right ]
\]
with the initial condition
\[
v_{\M_3}=\left[\begin{matrix} \eta^{\mathrm{i}}_1 \\ \eta^{\mathrm{ii}}_1 \\ \eta^{\mathrm{iii}}_1 \\ \eta^{\mathrm{iv}}_1 \end{matrix} \right ]=\left [ \begin{matrix} q \\ q^2-q \\ q \\ q^3 \end{matrix} \right ] \qquad \text{and} \qquad v_{\GL_3}=\left[\begin{matrix} \eta^{\mathrm{i}}_1 \\ \eta^{\mathrm{ii}}_1 \\ \eta^{\mathrm{iii}}_1 \\ \eta^{\mathrm{iv}}_1 \end{matrix} \right ]=\left [ \begin{matrix} q-1 \\ (q-1)(q-2) \\ q-1 \\ q^3-q^2 \end{matrix} \right ]
\]
Thus, if $\epsilon=(1,1,1,1)$, we have
\begin{equation}\label{s}
\begin{split}
s_{\M_3}(\ii)&=|\text{Sim}(\M_3(A_\ii))|=\epsilon T^{\ii-1} v_{\M_3}, \\
s_{\GL_3}(\ii)&=|\text{Sim}(\GL_3(A_\ii))|=\epsilon T^{\ii-1} v_{\GL_3}.
\end{split}
\end{equation}
A straightforward induction on $\ii$ gives the following formula for $T^{\ii}$.
\begin{claim}\label{T}
\[
T^{\ii}= \left [ \begin{matrix} q^\ii & 0 & 0 & 0 \\ q^{2\ii}-q^\ii & q^{2\ii} & 0 & 0 \\ q^{\ii}\frac{q^\ii-1}{q-1} & 0 & q^{2\ii} & 0 \\ \theta_\ii & q^{2\ii+1}\frac{q^\ii-1}{q-1}  & q^{2\ii-1}(q^2+1)\frac{q^\ii-1}{q-1}  & q^{3\ii} \end{matrix} \right ]
\]
where $\theta_\ii$ is given by
\[
\theta_\ii=q^{\ii-1}\left(\frac{q^\ii-1}{q-1}\right)\left[\left(\frac{q^4+1}{q-1}\right)\left(\frac{q^{\ii}+1}{q+1}\right)-
\left(\frac{q^3+1}{q-1}\right)\right]
\]
\end{claim}
\bigskip
\bigskip

\noindent Combining Claim \ref{T} and \eqref{s} gives
\begin{theorem} The number of similarity classes of matrices in $\M_3(A _\ii)$ is
\[
s_{\M_3}(\ii)=\frac{q^{3\ii+3}+q^{3\ii-1}-q^{2\ii+2}-q^{2\ii+1}-q^{2\ii}-q^{2\ii-1}+2q^\ii}{(q-1)(q^2-1)}.
\]
The number of similarity classes of elements in $\GL_3(A _\ii)$ is
\[
s_{\GL_3}(\ii)=\frac{q^{3\ii+2}-q^{3\ii}+2q^{3\ii-2}-q^{2\ii+1}-q^{2\ii-1}-2q^{2\ii-2}+2q^{\ii-1}}{q^2-1}.
\]
\end{theorem}
These can be packed in terms of generating functions
\[
\begin{split}
\mathcal{Z}_{\M_3}(z)&=\sum_{\ii=0}^{\infty} s_{\M_3}(\ii)z^\ii =\frac{1}{(q-1)(q^2-1)} \left ( \frac{q^3+q^{-1}}{1-q^3z} -\frac{q^2+q+1+q^{-1}}{1-q^2z} +\frac{2}{1-qz} \right ) \\
 \mathcal{Z}_{\GL_3}(z)&=\sum_{\ii=0}^{\infty} s_{\GL_3}(\ii)z^\ii=\frac{1}{q^2-1}\left(\frac{q^2-1+2q^{-2}}{1-q^3z} - \frac{q+q^{-1}+2q^{-2}}{1-q^2z}+\frac{2q^{-1}}{1-qz} \right).
\end{split}
\]

\providecommand{\bysame}{\leavevmode\hbox to3em{\hrulefill}\thinspace}
\providecommand{\MR}{\relax\ifhmode\unskip\space\fi MR }
\providecommand{\MRhref}[2]{%
  \href{http://www.ams.org/mathscinet-getitem?mr=#1}{#2}
}
\providecommand{\href}[2]{#2}

\vspace{\bigskipamount}

\begin{footnotesize}
\begin{quote}

Nir Avni\\
Einstein Institute of Mathematics,\\
The Hebrew University of Jerusalem, Edmond Safra Campus, Givat Ram, \\
 Jerusalem 91904, Israel \\
{\tt  avni.nir@gmail.com} \\

Uri Onn\\
Ben-Gurion university of the Negev, \\
Beer-Sheva 84105, Israel\\
{\tt urionn@math.bgu.ac.il} \\

Amritanshu Prasad\\
The Institute of Mathematical Sciences, CIT campus, \\
Chennai 600 113, India \\
{\tt amri@imsc.res.in} \\

Leonid Vaserstein\\
Department of Mathematics, Penn State University, University Park PA
\\ 16802-6401, USA \\
{\tt vstein@math.psu.edu}

\end{quote}
\end{footnotesize}

\end{document}